\ifdefined\pdfoutput\pdfoutput=1\fi
\documentclass[11pt,a4paper]{article}

\usepackage[margin=28mm]{geometry}
\usepackage[T1]{fontenc}
\usepackage[utf8]{inputenc}
\usepackage{lmodern}
\usepackage{microtype}
\usepackage{amsmath,amssymb,amsthm,mathtools}
\usepackage{enumitem}
\usepackage{authblk}
\usepackage[hidelinks]{hyperref}
\setlist[enumerate]{label=\arabic*., leftmargin=2em}
\setlist[itemize]{leftmargin=2em}

\newtheorem{theorem}{Theorem}[section]
\newtheorem{lemma}[theorem]{Lemma}
\newtheorem{proposition}[theorem]{Proposition}
\newtheorem{corollary}[theorem]{Corollary}
\theoremstyle{definition}
\newtheorem{definition}[theorem]{Definition}

\theoremstyle{remark}
\newtheorem{remark}[theorem]{Remark}

\hypersetup{
  pdftitle={Banach Spaces Generated by Finite-Valued Functions: Superreflexive Rigidity, Hankel Operators, and Universality},
  pdfauthor={Miroslav Hristov, Atanas Ilchev, Diana Nedelcheva, Boyan Zlatanov}
}

\title{\Large\bfseries Banach Spaces Generated by Finite-Valued Functions:\\
	Superreflexive Rigidity, Hankel Operators, and Universality}
\author[1]{Miroslav Hristov}
\author[2]{Atanas Ilchev}
\author[3]{Diana Nedelcheva}
\author[2]{Boyan Zlatanov}
\affil[1]{Konstantin Preslavsky University of Shumen, 115 Universitetska Str., Shumen, 9700, Bulgaria.}
\affil[2]{University of Plovdiv Paisii Hilendarski, 24 Tsar Assen Str., Plovdiv, 4000, Bulgaria.}
\affil[3]{Technical University of Varna, 1 Studentska Str., Varna, 9000, Bulgaria.}
\date{}

\begin{document}
\maketitle

\begin{abstract}
	We study Banach spaces generated by uniformly bounded finite-valued
	functions and establish a rigidity phenomenon for their geometry.
	Let \(F\subset\mathbb R\) be finite, let
	\(\mathcal A\subset F^\Omega\), and set
	\(X_{\mathcal A}
	=
	\overline{\operatorname{span}}^{\,\|\cdot\|_\infty}\mathcal A
	\subseteq\ell_\infty(\Omega)\).
	We prove that \(X_{\mathcal A}\) is superreflexive if and only if
	\(\mathcal A\) is finite, equivalently, if and only if
	\(X_{\mathcal A}\) is finite-dimensional. The main ingredient is a
	finite-range spreading-model obstruction showing that a normalized
	weakly null sequence of finite-valued functions admits a subsequence
	whose spreading model has nonreflexive closed linear span.
	
	We apply this rigidity principle to the Banach space \(X_L\) generated by
	the characteristic functions of the left derivatives of a formal
	language \(L\). As a consequence, \(L\) is regular if and only if
	\(X_L\) is finite-dimensional, if and only if \(X_L\) is superreflexive,
	or equivalently, admits an equivalent uniformly convex norm. We further
	obtain a canonical dynamical representation of \(X_L\) as a
	coordinate-function subspace of \(C(K_L)\), where \(K_L\) is the
	associated compact shift-orbit closure.
	
	The corresponding language Hankel operator
	\(\mathcal H_L:\ell_1(\Sigma^*)\to\ell_\infty(\Sigma^*)\) is compact if
	and only if \(L\) is regular. For nonregular languages its distance from
	both the compact and finite-rank operators is exactly \(1/2\), and
	\(a_n(\mathcal H_L)=1/2\) for every \(n\geq2\). Finally, we construct a
	single binary language \(L_{\mathrm{univ}}\) such that
	\(X_{L_{\mathrm{univ}}}\) contains an isometric copy of \(C(\Delta)\);
	consequently, it contains an isometric copy of every separable real
	Banach space.
\end{abstract}

{\small
\par}

\section{Introduction}\label{sec:introduction}

A recurring theme in Banach space theory is the extent to which the
geometric properties of a closed linear span are determined by the
structure of its generating family. This question becomes particularly
rigid when the generators are functions taking values in a fixed finite
set. Let \(\Omega\) be a nonempty set, let \(F\subseteq\mathbb R\) be
finite, and let
\(\mathcal A\subseteq F^\Omega\).

The associated Banach space
\(X_{\mathcal A}
=
\overline{\operatorname{span}\mathcal A}^{\,\|\cdot\|_\infty}
\subseteq \ell_\infty(\Omega)\)
may have a rich linear structure even though every member of the
generating family has finite range. The interaction between this
discrete structure of the generators and the geometry of
\(X_{\mathcal A}\) is the starting point of the present paper.

Our principal result establishes a strong rigidity phenomenon for such
spaces. We prove that
\[
\mathcal A\text{ is finite}
\Longleftrightarrow
X_{\mathcal A}\text{ is finite-dimensional}
\Longleftrightarrow
X_{\mathcal A}\text{ is superreflexive}.
\]

By the classical renorming theory of superreflexive spaces, the latter
condition is also equivalent to the existence of an equivalent uniformly
convex norm \cite{Enflo1972,James1972,Pisier1975}. Thus an infinite family
of functions taking values in a common finite subset of \(\mathbb R\)
can never generate a superreflexive Banach space when the closed linear
span is equipped with the supremum norm. Notice that no algebraic,
dynamical, or invariance assumption is imposed on the family
\(\mathcal A\).

The main ingredient in the proof is an obstruction obtained through
spreading models. Starting from an infinite finite-valued generating
family, reflexivity would produce a normalized weakly null sequence whose
elements still take values in a fixed finite set. A Ramsey stabilization
argument is then applied to its finite coordinate-pattern sets. The
resulting stabilized patterns determine a compact set
\(K\subseteq A^{\mathbb N}\)
which represents the spreading norm exactly. This realization makes it
possible to prove that the closed linear span of the resulting spreading
model is nonreflexive. Since spreading models are finitely representable
in the original space, this contradicts superreflexivity. The argument
combines classical ideas concerning spreading models
\cite{ArgyrosKanellopoulosTyros2013,BrunelSucheston1974}, finite
representability, and spaces of continuous functions on countable compact
sets \cite{BessagaPelczynski1960,PelczynskiSemadeni1959}.

A natural and particularly structured source of finite-valued families
arises from formal languages. Let \(\Sigma\) be a finite alphabet,
\(\Sigma^*\) the free monoid of finite words, and
\(L\subseteq\Sigma^*\) a language. For \(u\in\Sigma^*\), the left
quotient, or left derivative, of \(L\) by \(u\) is
\(u^{-1}L=\{v\in\Sigma^*:uv\in L\}\).
Representing this set by its characteristic function gives
\(q_u^L(v)
=
\mathbf 1_L(uv)\) for
\(v\in\Sigma^*\).
Hence every \(q_u^L\) is a \(\{0,1\}\)-valued element of
\(\ell_\infty(\Sigma^*)\). 

We associate with \(L\) the Banach space
\(X_L
=
\overline{\operatorname{span}}
\{q_u^L:u\in\Sigma^*\}^{\,\|\cdot\|_\infty}\).

Although the generators of \(X_L\) are Boolean-valued functions, the
space \(X_L\) itself is their closed real linear span and therefore need
not consist of Boolean-valued elements. Since \(\Sigma^*\) is countable,
\(X_L\) is always separable.

This construction connects the geometry of Banach spaces with classical
finite-state phenomena. The Myhill--Nerode theorem characterizes
regular languages by the finiteness of their families of left quotients
\cite{Eilenberg1974,HopcroftUllman1979,Myhill1957,Nerode1958}.
Distinct Boolean derivatives are separated by distance exactly one in
the supremum norm. Consequently,
\[
L\text{ is regular}
\Longleftrightarrow
\{q_u^L:u\in\Sigma^*\}\text{ is finite}
\Longleftrightarrow
X_L\text{ is finite-dimensional}.
\]

Applying the general finite-valued rigidity theorem strengthens this
classical correspondence to the next equivalent conditions
\begin{itemize}
	\item \(L\) is regular
    \item \(X_L\) is superreflexive
    \item \(X_L\) admits an equivalent uniformly convex norm.
\end{itemize}

Thus regularity, which is originally a discrete finite-state property,
admits a characterization in terms of the infinite-dimensional geometry
of a canonically associated Banach space.

The derivative-generated space also admits natural dynamical and
operator-theoretic representations. Identifying \(L\) with its
characteristic point
\(x_L=\mathbf 1_L\in\{0,1\}^{\Sigma^*}\),
the free monoid acts by shifts
\((\rho_vx)(u)=x(uv)\).

The orbit closure of \(x_L\) is a compact shift-invariant space
\(K_L\), and \(X_L\) is linearly isometric to the closed subspace of
\(C(K_L)\) generated by the corresponding coordinate functions.
Conversely, every compact binary shift system with a transitive point
arises from a language in this way. This connects the construction with
standard ideas from topological and symbolic dynamics
\cite{Ellis1969,GlasnerMegrelishvili2014,Kitchens1998}.

There is also a canonical Hankel-operator realization. Define
\(\mathcal H_L:\ell_1(\Sigma^*)\longrightarrow\ell_\infty(\Sigma^*)\)
by
\((\mathcal H_La)(v)=\sum_{u\in\Sigma^*}a(u)\mathbf 1_L(uv)\).

Its values on the standard unit vectors are precisely the Boolean
derivatives \(q_u^L\), and
\(\overline{\operatorname{ran}\mathcal H_L}=X_L\).
The Hankel representation is closely related to the classical
linear-algebraic approach to formal languages and weighted automata
\cite{BerstelReutenauer2011,CarlylePaz1971,Fliess1974}. We prove that
\(L\) is regular if and only if \(\mathcal H_L\) has finite rank, and
equivalently if and only if \(\mathcal H_L\) is compact. For a
nonregular language, this gives the exact distance formula
\(\operatorname{dist}(\mathcal H_L,\mathcal F)=\operatorname{dist}(\mathcal H_L,\mathcal K)=\frac12\),
where \(\mathcal F\) and \(\mathcal K\) denote the finite-rank and compact
operators, respectively. 

Moreover, the approximation numbers satisfy
\(a_n(\mathcal H_L)=\frac12\) for \(n\geq2\).

At the opposite extreme from the rigidity phenomenon, we establish a
universality result. Using an explicit prefix-free binary coding, we
construct a language
\(L_{\mathrm{univ}}\subseteq\{0,1\}^*\)
for which \(X_{L_{\mathrm{univ}}}\) contains a linearly isometric copy of
\(C(\Delta)\), where
\(\Delta=\{0,1\}^{\mathbb N}\)
is the Cantor space. Since every separable real Banach space admits a
linear isometric embedding into \(C(\Delta)\), it follows that a single
derivative-generated space contains isometric copies of all separable
real Banach spaces. Thus the class considered here exhibits two
apparently opposite features: it is universal with respect to separable
subspace embeddings, while it is rigid with respect to the
superreflexivity of the whole space.

The paper is organized as follows.
Section~\ref{sec:preliminaries} collects the required background on
Banach-space geometry, finite representability, spreading models,
compact spaces, and formal languages.
Section~\ref{sec:main-results} states the main superreflexive rigidity
theorem for Banach spaces generated by finite-valued functions and its
principal consequence for formal languages.
Section~\ref{sec:auxiliary-results} establishes the auxiliary results
needed for the proof, including the finite-range spreading-model
obstruction and the finite representability of spreading models.
The proofs of the main results are given in
Section~\ref{sec:proof-main-results}.
Section~\ref{sec:formal-language-applications} develops the structural
applications to derivative-generated Banach spaces of formal languages.
Section~\ref{sec:dynamical-representation} gives the canonical dynamical
realization in terms of compact shift systems and coordinate-function
spaces.
Section~\ref{sec:hankel-operators} develops the Hankel-operator
representation and establishes the finite-rank, compactness, and exact
approximation-gap results.
Finally, Section~\ref{sec:universality} proves the universality theorem
and shows that a Banach space generated by a single binary language can
contain an isometric copy of every separable real Banach space.

\section{Preliminaries}\label{sec:preliminaries}

Throughout the paper, \(\mathbb N\) will denote the set of naturla numbers and \(\mathbb N_0=\mathbb{N}\cup\{0\}\). All considered Banach spaces are real ones.

\subsection{Banach-space preliminaries}

For a nonempty set \(\Omega\), let \(\ell_\infty(\Omega)\) denote the
Banach space of all bounded real-valued functions on \(\Omega\), equipped
with the norm
\(\|f\|_\infty=\sup_{\omega\in\Omega}|f(\omega)|\).

The space \(\ell_1(\Omega)\) consists of all absolutely summable families
\(a=(a(\omega))_{\omega\in\Omega}\), with
\(\|a\|_1=\sum_{\omega\in\Omega}|a(\omega)|\).

For \(\omega\in\Omega\), the corresponding standard unit vector in
\(\ell_1(\Omega)\) is denoted by \(e_\omega\).

If \(K\) is a compact Hausdorff space, \(C(K)\) denotes the Banach space
of all continuous real-valued functions on \(K\), equipped with the
supremum norm.

For Banach spaces \(X\) and \(Y\), we denote by
\(\mathcal L(X,Y)\), \(\mathcal F(X,Y)\), and \(\mathcal K(X,Y)\)
the spaces of bounded, finite-rank, and compact linear operators,
respectively. If \(T\in\mathcal L(X,Y)\) and
\(\mathcal M\subseteq\mathcal L(X,Y)\), we write
\(\operatorname{dist}(T,\mathcal M)
=
\inf_{S\in\mathcal M}\|T-S\|\).

For every \(n\in\mathbb N\), the \(n\)-th approximation number of
\(T\in\mathcal L(X,Y)\) is defined by
\(a_n(T)=\inf\{\|T-R\|:R\in\mathcal L(X,Y),\ \operatorname{rank}R<n\}\).
We refer to \cite{Pietsch1987} for the standard theory of approximation
numbers.

Recall that a Banach space \(Y\) is finitely representable in a Banach
space \(X\) if, for every finite-dimensional subspace \(E\subseteq Y\)
and every \(\varepsilon>0\), there exist a finite-dimensional subspace
\(F\subseteq X\) and an isomorphism \(T:E\to F\) such that
\(\|T\|\,\|T^{-1}\|<1+\varepsilon\).

A Banach space \(X\) is superreflexive if every Banach space finitely
representable in \(X\) is reflexive. In particular, every superreflexive
Banach space is reflexive. We shall also use the classical fact that \(X\) is superreflexive if and
only if it admits an equivalent uniformly convex norm
\cite{Enflo1972,James1972,Pisier1975}.

The Eberlein--Šmulian theorem will be a key one in our investigations:

\begin{theorem}[Eberlein--Šmulian, \cite{Eberlein1947,Smulian1940,Whitley1967}]\label{thm:eberlein-smulian}
	A Banach space \(X\) is reflexive if and only if every bounded sequence
	in \(X\) has a weakly convergent subsequence.
\end{theorem}

A sequence \((b_n)\) in a Banach space is called a basic sequence if it
is a Schauder basis for its closed linear span
\(B=\overline{\operatorname{span}}\{b_n:n\in\mathbb N\}\).
The associated biorthogonal functionals
\((b_n^*)\subseteq B^*\) are determined by
\(b_n^*(b_m)=\delta_{nm}\) for every pair of indices \(m,n\in\mathbb N\).

Equivalently, if
\(x=\sum_{j=1}^{\infty}a_jb_j\in B\),
then \(b_n^*(x)=a_n\).

\begin{definition}\label{def:spreading-model}
	Let \(\{x_n\}_{n=1}^\infty\) be a sequence in a Banach space \(X\), and let
	\(\{x_{m_j}\}_{j=1}^\infty\) be a subsequence. A basic sequence \(b_i\}_{i=1}^\infty\) is said to be
	a spreading model generated by \(\{x_{m_j}\}_{j=1}^\infty\) if, for every
	\(k\in\mathbb N\) and every choice of scalars \(a_1,\ldots,a_k\),
	\[
	\left\|
	\sum_{i=1}^{k}a_i b_i
	\right\|
	=
	\lim_{\substack{j_1<\cdots<j_k\\ j_1\to\infty}}
	\left\|
	\sum_{i=1}^{k}a_i x_{m_{j_i}}
	\right\|.
	\]
	
	Equivalently, for every \(\varepsilon>0\) there exists \(N\in\mathbb N\)
	such that
	\[
	\left|
	\left\|
	\sum_{i=1}^{k}a_i x_{m_{j_i}}
	\right\|
	-
	\left\|
	\sum_{i=1}^{k}a_i b_i
	\right\|
	\right|
	<\varepsilon
	\]
	whenever \(N\leq j_1<\cdots<j_k\).
\end{definition}

The following classical result will be used in the proof of the main
rigidity theorem.

\begin{theorem}[Brunel--Sucheston, \cite{ArgyrosKanellopoulosTyros2013,BrunelSucheston1974}]\label{thm:brunel-sucheston}
	Every normalized weakly null sequence admits a subsequence generating a
	normalized spreading model. A spreading model generated by a normalized
	weakly null sequence is a suppression-\(1\)-unconditional basic sequence.
\end{theorem}

\subsection{Compactness and \(C(K)\)-spaces}

We shall use the infinite Ramsey theorem in its classical form.

\begin{theorem}[Infinite Ramsey theorem, \cite{Ramsey1930}]\label{thm:ramsey}
	Let \(k,r\in\mathbb N\). If the \(k\)-element subsets of an infinite set
	are colored with \(r\) colors, then there exists an infinite subset all
	of whose \(k\)-element subsets have the same color.
\end{theorem}

The following property of spaces of continuous functions on countable
compact spaces will be used to detect nonreflexivity.

\begin{theorem}[\cite{BessagaPelczynski1960,PelczynskiSemadeni1959}]\label{thm:pelczynski-semadeni}
	Let \(K\) be a countable compact Hausdorff space. Then every
	infinite-dimensional closed subspace of \(C(K)\) contains a subspace
	isomorphic to \(c_0\). Consequently, every such subspace is nonreflexive.
\end{theorem}

\begin{theorem}[Tychonoff's theorem]\label{thm:tychonoff}
	An arbitrary product of compact topological spaces is compact with
	respect to the product topology.
\end{theorem}

In particular, if \(A\) is a finite discrete space, then
\(A^{\mathbb N}\), equipped with the product topology, is compact by
Theorem~\ref{thm:tychonoff}. Since the product is countable,
\(A^{\mathbb N}\) is also metrizable. Hence, every closed subset of
\(A^{\mathbb N}\) is compact and metrizable.

Let
\(\Delta=\{0,1\}^{\mathbb N}\)
denote the Cantor space. We shall use the following two classical facts.

\begin{theorem}[Density of clopen simple functions, \cite{Stone1948}]
	\label{thm:clopen-density}
	The linear span of the characteristic functions of the clopen subsets of
	\(\Delta\) is uniformly dense in \(C(\Delta)\).
\end{theorem}

\begin{theorem}[Cantor-space form of the Banach--Mazur theorem, \cite{Alexandroff1927,Banach1932}]
	\label{thm:banach-mazur}
	Every separable real Banach space admits a linear isometric embedding into
	\(C(\Delta)\).
\end{theorem}

\subsection{Formal-language preliminaries}

Let \(\Sigma\) be a finite nonempty alphabet. The free monoid generated by
\(\Sigma\) is denoted by \(\Sigma^*\), its identity element is the empty
word \(\varepsilon\), and its operation is concatenation. A formal
language over \(\Sigma\) is a subset \(L\subseteq\Sigma^*\).

The set \(\Sigma^*\), equipped with concatenation and with the empty
word \(\varepsilon\) as its identity element, is the free monoid
generated by \(\Sigma\). Concatenation in \(\Sigma^*\) satisfies the
left and right cancellation laws:
\(uw=uv\ \Longrightarrow\ w=v\), and \(wu=vu\ \Longrightarrow\ w=v\).

For a set \(B\), its characteristic function is denoted by
\(\mathbf 1_B\).

\begin{definition}\label{def:left-quotient}
	Let \(L\subseteq\Sigma^*\) and \(u\in\Sigma^*\). The left quotient, or
	left derivative, of \(L\) by \(u\) is
	\(u^{-1}L
	=
	\{v\in\Sigma^*:uv\in L\}\).
	The family of all left derivatives of \(L\) is
	\(\mathcal D(L)=\{u^{-1}L:u\in\Sigma^*\}\).
\end{definition}

The fundamental finite-state characterization is the Myhill--Nerode
theorem.

\begin{theorem}[Myhill--Nerode, \cite{Eilenberg1974,HopcroftUllman1979,Myhill1957,Nerode1958}]\label{thm:myhill-nerode}
	For a language \(L\subseteq\Sigma^*\), the following are equivalent:
	\begin{enumerate}
		\item \(L\) is regular;
		\item the derivative family \(\mathcal D(L)\) is finite;
		\item the relation \(\equiv_L\) has finite index, where \(u\equiv_L v\)
		means that for every \(w\in\Sigma^*\), \(uw\in L\) if and only if \(vw\in L\).
	\end{enumerate}
	Moreover, the number of equivalence classes equals the number of states
	of the minimal complete deterministic automaton recognizing \(L\).
\end{theorem}

\section{Main Results}\label{sec:main-results}

We first state the main geometric result in an abstract form. Its
formulation is independent of formal languages and concerns Banach spaces
generated by functions taking values in a fixed finite set.

\begin{theorem}\label{thm:main-rigidity}
	Let \(\Omega\) be a set, let \(F\subset \mathbb{R}\) be finite, and let
	\(\mathcal A\subset F^\Omega\). Define
	\(X_{\mathcal A}=\overline{\operatorname{span}}^{\,\|\cdot\|_\infty}\mathcal A\subset\ell_\infty(\Omega)\).
	Then the following assertions are equivalent:
	\begin{enumerate}
		\item \(\mathcal A\) is finite;
		\item \(X_{\mathcal A}\) is finite-dimensional;
		\item \(X_{\mathcal A}\) is superreflexive;
		\item \(X_{\mathcal A}\) admits an equivalent uniformly convex norm.
	\end{enumerate}
	In particular, an infinite family of uniformly bounded finite-valued
	functions cannot generate an infinite-dimensional superreflexive subspace
	of \(\ell_\infty(\Omega)\).
\end{theorem}

As a first consequence, the preceding rigidity theorem gives a complete
geometric characterization of regular languages.

\begin{corollary}\label{cor:main-language-rigidity}
	Let \(L\subseteq\Sigma^*\) be a language over a finite alphabet, and let
	\(X_L=\overline{\operatorname{span}}\{\mathbf 1_{w^{-1}L}:w\in\Sigma^*\}\subset\ell_\infty(\Sigma^*)\).
	Then the following assertions are equivalent:
	\begin{enumerate}
		\item \(L\) is regular;
		\item \(X_L\) is finite-dimensional;
		\item \(X_L\) is superreflexive;
		\item \(X_L\) admits an equivalent uniformly convex norm.
	\end{enumerate}
\end{corollary}

\section{Auxiliary Results}\label{sec:auxiliary-results}

Throughout the paper, \(x_n\rightharpoonup x\) denotes that the sequence
\(\{x_n\}\) converges weakly to \(x\).

\begin{lemma}\label{lem:finite-valued-separated}
	Let \(F\subset\mathbb{R}\) be finite and let
	\(\mathcal A\subset F^\Omega\). If \(\mathcal A\) is infinite, then
	\(\mathcal A\), regarded as a subset of \(\ell_\infty(\Omega)\), is not
	relatively compact.
\end{lemma}

\begin{proof}
	Since \(F\) is finite, the minimum 
	\(\delta_F
	=
	\min\{|a-b|:a,b\in F,\ a\neq b\}>0\) exists,
	provided that \(F\) contains at least two elements. If \(f,g\in\mathcal A\)
	are distinct, then \(f(\omega)\neq g(\omega)\) for some
	\(\omega\in\Omega\), and therefore
	\(\|f-g\|_\infty\geq\delta_F\).
	
	Thus \(\mathcal A\) is an infinite \(\delta_F\)-separated subset of
	\(\ell_\infty(\Omega)\), and hence it is not relatively compact.
\end{proof}

\begin{corollary}\label{cor:finite-dimensional-family}
	If
	\(X_{\mathcal A}
	=
	\overline{\operatorname{span}}\mathcal A\)
	is finite-dimensional, then \(\mathcal A\) is finite.
\end{corollary}

\begin{proof}
	Suppose contrary, i.e., \(\mathcal A\) is infinite. Since
	\(\mathcal A\subset F^\Omega\), every \(f\in\mathcal A\) satisfies
	\(\|f\|_\infty\leq\max\{|a|:a\in F\}\).
	
	Hence \(\mathcal A\) is a bounded subset of \(X_{\mathcal A}\).
	
	Since \(X_{\mathcal A}\) is finite-dimensional, every bounded subset of
	\(X_{\mathcal A}\) is relatively compact. Thus \(\mathcal A\) is
	relatively compact in \(X_{\mathcal A}\), and therefore also in
	\(\ell_\infty(\Omega)\).
	
	This contradicts Lemma~\ref{lem:finite-valued-separated}, which states
	that an infinite subset of \(F^\Omega\) is not relatively compact.
	Consequently, \(\mathcal A\) must be finite.
\end{proof}

\begin{lemma}\label{lem:weak-limit-finite-valued}
	Let \(F\subset\mathbb R\) be finite and let
	\((f_n)\subset F^\Omega\) converge weakly in
	\(\ell_\infty(\Omega)\) to \(f\). Then
	\(f\in F^\Omega\).
\end{lemma}

\begin{proof}
	Fix \(\omega\in\Omega\). The evaluation functional
	\(\delta_\omega:\ell_\infty(\Omega)\to\mathbb R\), defined by
	\(\delta_\omega(g)=g(\omega)\),
	is bounded and linear. Hence the weak convergence
	\(f_n\rightharpoonup f\) implies
	\(f_n(\omega)
	=
	\delta_\omega(f_n)
	\longrightarrow
	\delta_\omega(f)
	=
	f(\omega)\).
	
	Since \(f_n(\omega)\in F\) for every \(n\), and \(F\) is finite and
	therefore closed in \(\mathbb R\), it follows that
	\(f(\omega)\in F\).
	
	As \(\omega\in\Omega\) was arbitrary, we conclude that
	\(f\in F^\Omega\).
\end{proof}

For subsets \(A,B\) of a vector space, we use the notation
\(A-B:=\{a-b:a\in A,\ b\in B\}\)
for their algebraic difference.

\begin{lemma}\label{lem:normalized-difference-sequence}
	Let \(F\subset\mathbb R\) be finite and let
	\((f_n)\subset F^\Omega\) be a sequence of pairwise distinct functions
	converging weakly to \(f\). Then, after passing to a subsequence, there
	exists \(c>0\) such that
	\(x_n=\frac{f_n-f}{c}\) for \(n\in\mathbb{N}\)
	is a normalized weakly null sequence whose elements take values in a fixed
	finite subset of \(\mathbb R\).
\end{lemma}

\begin{proof}
	By Lemma~\ref{lem:weak-limit-finite-valued}, the weak limit \(f\)
	belongs to \(F^\Omega\). Hence, for every \(n\) and every
	\(\omega\in\Omega\), the inclusion
	\(f_n(\omega)-f(\omega)\in F-F\) holds,
	where
	\(F-F=\{a-b:a,b\in F\}\)
	is a finite subset of \(\mathbb R\).
	
	Since the functions \(f_n\) are pairwise distinct, at most one of them
	can coincide with \(f\). Discarding this term if necessary, we may assume
	that
	\(f_n-f\neq 0\) for every \(n\in\mathbb{N}\).
	
	For each \(n\), there holds
	\(\|f_n-f\|_\infty
	=
	\sup_{\omega\in\Omega}|f_n(\omega)-f(\omega)|\).
	
	As the values \(f_n(\omega)-f(\omega)\) belong to the finite set
	\(F-F\), this supremum is attained as one of the finitely many positive
	values in
	\(\{|a-b|:a,b\in F,\ a\neq b\}\).
	
	Consequently, the sequence
	\(\bigl\{\|f_n-f\|_\infty\bigr\}_{n=1}^\infty\)
	takes only finitely many positive values. 
	
	Therefore, after passing to an
	infinite subsequence \(\{f_{n_k}\}_{k=1}^\infty\), which just for simplicity of the notation we will denote by \(\{f_n\}\), there exists \(c>0\) such that
	\(\|f_n-f\|_\infty=c\)
	for every \(n\) in this subsequence.
	
	Define
	\(x_n=\frac{f_n-f}{c}\).
	Then
	\(\|x_n\|_\infty=1\),
	so \((x_n)\) is normalized. 
	
	Moreover, since
	\(f_n\rightharpoonup f\), we have
	\(_n\rightharpoonup 0\),
	and hence \((x_n)\) is weakly null.
	
	Finally, for every \(n\) and \(\omega\in\Omega\),
	\(x_n(\omega)\in
	A=\frac{1}{c}(F-F)\),
	where \(A\subset\mathbb R\) is finite and \(0\in A\). Thus
	\((x_n)\) is a normalized weakly null sequence whose elements take
	values in the fixed finite set \(A\).
\end{proof}

A basic sequence \((b_n)\) is called
suppression-\(1\)-unconditional if, for every finitely supported scalar
sequence \((a_n)\) and every set \(A\subseteq\mathbb N\),
\(\|\sum_{n\in A}a_nb_n\|\leq\|\sum_{n=1}^{\infty}a_nb_n\|\).

A spreading model generated by a normalized weakly null sequence is a
suppression-\(1\)-unconditional basic sequence.

\begin{lemma}\label{lem:finite-range-spreading}
	Let \(A\subset\mathbb R\) be finite with \(0\in A\), and let
	\(X\subset\ell_\infty(\Omega)\) be a closed subspace. If
	\(\{x_n\}_{n=1}^\infty\subset X\) is a normalized weakly null sequence, so that
	\(x_n(\omega)\in A\) for \((n,\omega)\in\mathbb N\times\Omega\),
	then there is a subsequence \(\{x_{n_k}\}_{k=1}^\infty\) which generates a spreading model whose
	closed linear span is nonreflexive.
\end{lemma}

\begin{proof}
	According to Theorem \ref{thm:brunel-sucheston}, after passing to a subsequence and just to simplify the notations denoting it by \(\{x_n\}_{n=1}^\infty\),
	we may assume that \(\{x_n\}_{n=1}^\infty\) generates a normalized spreading model
	\(\{b_n\}_{n=1}^\infty\). Since the subsequence \(\{x_n\}_{n=1}^\infty\) is weakly null, \(\{b_n\}\) is a
	suppression-\(1\)-unconditional basic sequence.
	
	For \(k\in\mathbb N\) and \(n_1<\cdots<n_k\), define the pattern set
	\[
	P(n_1,\ldots,n_k)
	=
	\left\{
	\bigl(x_{n_1}(\omega),\ldots,x_{n_k}(\omega)\bigr):
	\omega\in\Omega
	\right\}
	\subset A^k.
	\]
	
	Since \(A\) is finite, there are at most
	\(2^{|A|^k}\)
	possible pattern sets for a fixed \(k\). According to Theorem \ref{thm:ramsey},
	and applying a diagonal selection, we may pass to a further subsequence
	so that, for every \(k\), the pattern set stabilizes to a set
	\(P_k\subset A^k\), i.e., \(P(n_{j_1},\ldots,n_{j_k})=P_k\), 
	whenever \(j_1<\cdots<j_k\) are sufficiently large. Passing to a
	further subsequence does not change the spreading model.
	
	Consequently, for every \(k\in\mathbb N\) and every scalars
	\(a_1,\ldots,a_k\),
	\begin{equation}\label{eq:spreading-pattern-norm}
	\left\|\sum_{i=1}^k a_i b_i\right\|
	=
	\max_{(\alpha_1,\ldots,\alpha_k)\in P_k}
	\left|
	\sum_{i=1}^k a_i\alpha_i
	\right|.
\end{equation}
	
	The family \((P_k)\) is compatible under coordinate restriction.
	Define
	\[
	K
	=
	\left\{
	\eta=(\eta_i)_{i=1}^\infty\in A^{\mathbb N}:
	(\eta_1,\ldots,\eta_k)\in P_k
	\text{ for every }k\in\mathbb N
	\right\}.
	\]
	
	Since \(A\) is finite and discrete, the set
	\(\bigl\{\eta\in A^{\mathbb N}:
	(\eta_1,\ldots,\eta_k)\in P_k\bigr\}\)
	is  is both open and
	closed subset of \(A^{\mathbb N}\). Hence \(K\) is closed
	in the compact space \(A^{\mathbb N}\), and therefore \(K\) is compact.
	Moreover, compatibility of the sets \(P_k\) implies that every element
	of \(P_k\) occurs as the first \(k\) coordinates of some element of
	\(K\).
	
	For \(i\in\mathbb N\), let the functional
	\(p_i:K\to\mathbb R\) be defined by \(p_i(\eta)=\eta_i\).
	
	It follows from \eqref{eq:spreading-pattern-norm} that
	\(\|\sum_{i=1}^k a_i b_i\|=\|\sum_{i=1}^k a_i p_i\|_{C(K)}\).
	
	Thus the correspondence
	\(b_i\longmapsto p_i\)
	extends to an isometric embedding
	\(T:E\longrightarrow C(K)\), where \(E=\overline{\operatorname{span}}\{b_i:i\in\mathbb N\}\).
	
	We now distinguish two cases: 1) \(K\) contains an element
	\(\eta=\{\eta_i\}\) with infinite support and 2) every \(\eta\in K\) has finite support.
	
	Case 1)
	Suppose that \(K\) contains an element
	\(\eta=\{\eta_i\}\) with infinite support. Since \(A\) is finite, there
	exist \(\alpha\in A\setminus\{0\}\) and an increasing sequence
	\(i_j\) of naturals such that
	\(\eta_{i_j}=\alpha\) for every \(j\).
	
	Suppose, towards a contradiction, that \(E\) is reflexive. Then the
	bounded sequence \(\{b_{i_j}\}\) has a weakly convergent subsequence,
	say
	\(b_{i_{j_\ell}}\rightharpoonup x\in E\).
	
	Let \(b_n^*\) denote the biorthogonal functionals associated with the
	basic sequence \(\{b_n\}\). For every fixed \(n\) and for all sufficiently large \(\ell\) there holds
	\(b_n^*(b_{i_{j_\ell}})=0\).
	
	Hence
	\(b_n^*(x)=0\) for every \(n\), and therefore \(x=0\).
	
	On the other hand, evaluation at \(\eta\) defines a bounded linear
	functional on \(C(K)\). Hence
	\(\varphi_\eta
	=
	\delta_\eta\circ T
	\in E^*\).
	
	For every \(\ell\) there holds
	\(\varphi_\eta(b_{i_{j_\ell}})
	=
	p_{i_{j_\ell}}(\eta)
	=
	\eta_{i_{j_\ell}}
	=
	\alpha\neq0\), which contradicts with the weakly null convergence
	\(b_{i_{j_\ell}}\rightharpoonup0\). Thus \(E\) is nonreflexive.
	
	Case 2)
	Suppose that every \(\eta\in K\) has finite support. Since \(A\) is
	finite, the set of all finitely supported sequences in
	\(A^{\mathbb N}\) is countable. Hence \(K\) is a countable compact
	space.
	
	The space \(E\) is infinite-dimensional because \(\{b_n\}\) is a basic
	sequence. Consequently, \(K\) cannot be finite, since otherwise
	\(C(K)\) would be finite-dimensional. Moreover, \(T(E)\) is a closed
	infinite-dimensional subspace of \(C(K)\). By
	Theorem~\ref{thm:pelczynski-semadeni}, \(T(E)\) contains a subspace
	isomorphic to \(c_0\). Hence \(T(E)\), and therefore \(E\), is
	nonreflexive.
	
	In both cases the closed linear span of the spreading model
	\(\{b_n\}\) is nonreflexive, which completes the proof.
\end{proof}

\begin{proposition}\label{prop:spreading-finite-representability}
	Let \(\{x_n\}_{n=1}^\infty\) be a bounded sequence in a Banach space \(X\), and let
	\(\{x_{m_j}\}_{k=1}^\infty\) generates a basic spreading model \(\{b_n\}\). Then
	\(E=\overline{\operatorname{span}}\{b_n:n\in\mathbb N\}\)
	is finitely representable in \(X\).
\end{proposition}

\begin{proof}
	Let \(G\subset E\) be finite-dimensional and let \(\varepsilon>0\).
	We shall construct a linear isomorphism from \(G\) onto a subspace of
	\(X\) with distortion less than \(1+\varepsilon\).
	
	Let
	\(E_k=\operatorname{span}\{b_1,\ldots,b_k\}\),
	and denote by
	\(P_k:E\to E_k\)
	the canonical partial-sum projections associated with the basic
	sequence \(\{b_n\}_{n=1}^\infty\). Since \(P_kx\to x\) for every \(x\in E\), and the
	unit sphere of \(G\) is compact, the convergence is uniform on that
	sphere. 
	
	Hence, for every \(\eta>0\), there exists \(k\) such that
	\(\|P_kx-x\|\leq\eta\|x\|\) for \(x\in G\).
	
	Consequently, for every \(x\in G\),
	\begin{equation}\label{eq:Pk-G}
	(1-\eta)\|x\|
	\leq
	\|P_kx\|
	\leq
	(1+\eta)\|x\|.
\end{equation}
	
	In particular, if \(0<\eta<1\), then \(P_k|_G\) is an isomorphism
	from \(G\) onto the subspace \(P_k(G)\subset E_k\).
	
	We next use the fact that \(\{b_n\}_{n=1}^\infty\) is the spreading model generated
	by \(\{x_{m_j}\}_{j=1}^\infty\). For every scalars \(a_1,\ldots,a_k\),
	\[
	\left\|\sum_{i=1}^k a_i b_i\right\|
	=
	\lim_{\substack{j_1<\cdots<j_k\\j_1\to\infty}}
	\left\|
	\sum_{i=1}^k a_i x_{m_{j_i}}
	\right\|.
	\]
	
	Since \(E_k\) is finite-dimensional and \((x_n)\) is bounded, a
	standard finite-net argument makes this convergence uniform on the
	unit sphere of \(E_k\). Thus, for every \(\delta>0\), we can choose
	\(j_1<\cdots<j_k\) sufficiently large so that the map
	\(U_k:E_k\to X\), defined by
	\(U_k(\sum_{i=1}^k a_i b_i)=\sum_{i=1}^k a_i x_{m_{j_i}}\),
	satisfies, for every \(y\in E_k\), the estimate
	\begin{equation}\label{eq:Uk}
	(1-\delta)\|y\|
	\leq
	\|U_ky\|
	\leq
	(1+\delta)\|y\|.
\end{equation}
	
	Define
	\(T=U_k\circ P_k|_G:G\to X\).
	Combining \eqref{eq:Pk-G} and \eqref{eq:Uk} gives
	\((1-\delta)(1-\eta)\|x\|\leq\|Tx\|\leq(1+\delta)(1+\eta)\|x\|\) for every \(x\in G\).
	
	Hence
	\(\|T\|\,\|T^{-1}\|
	\leq
	\frac{(1+\delta)(1+\eta)}
	{(1-\delta)(1-\eta)}\).
	
	Choosing \(\delta,\eta>0\) sufficiently small, we may ensure that
	\(\frac{(1+\delta)(1+\eta)}
	{(1-\delta)(1-\eta)}
	<
	1+\varepsilon\).
	
	Therefore \(G\) is \((1+\varepsilon)\)-isomorphic to the
	finite-dimensional subspace \(T(G)\subset X\).
	Since \(G\subset E\) and \(\varepsilon>0\) were arbitrary, \(E\) is
	finitely representable in \(X\).
\end{proof}

\section{Proof of Main Results}
\label{sec:proof-main-results}

\subsection{Proof of Theorem \ref{thm:main-rigidity}}

Assume that \(X_{\mathcal A}\) is superreflexive and suppose, towards a
contradiction, that \(\mathcal A\) is infinite. Choose a sequence of
pairwise distinct elements \((f_n)\subset\mathcal A\).

Since \(X_{\mathcal A}\) is reflexive, the Eberlein--Smulian theorem yields,
after passing to a subsequence, weak convergence \(f_n\rightharpoonup f\).
By Lemmas~\ref{lem:weak-limit-finite-valued} and
\ref{lem:normalized-difference-sequence}, we obtain a normalized weakly
null sequence taking values in a fixed finite subset of \(\mathbb R\).

Lemma~\ref{lem:finite-range-spreading} provides a subsequence generating a
spreading model whose closed linear span \(E\) is nonreflexive. By
Proposition~\ref{prop:spreading-finite-representability}, \(E\) is finitely
representable in \(X_{\mathcal A}\).

This contradicts the superreflexivity of \(X_{\mathcal A}\).

\subsection{Proof of Corollary \ref{cor:finite-dimensional-family}}

\begin{proof}
	Let us put
	\(\mathcal A_L
	=
	\left\{
	\mathbf 1_{w^{-1}L}:w\in\Sigma^*
	\right\}\).
	Then
	\(\mathcal A_L\subset\{0,1\}^{\Sigma^*}\)
	and, by definition,
	\(X_L
	=
	\overline{\operatorname{span}}\mathcal A_L\).
	
	According to Theorem \ref{thm:myhill-nerode}, \(L\) is regular if and only if the
	family of left quotients
	\(\{w^{-1}L:w\in\Sigma^*\}\)
	is finite. Since two subsets of \(\Sigma^*\) are equal if and only if
	their characteristic functions are equal, this is equivalent to
	\(\mathcal A_L\) being finite.
	
	Applying Theorem~\ref{thm:main-rigidity} with
	\(\Omega=\Sigma^*\) and \(F=\{0,1\}\),
	we obtain the equivalences
	\[
	\mathcal A_L\text{ is finite}
	\Longleftrightarrow
	X_L\text{ is finite-dimensional}
	\Longleftrightarrow
	X_L\text{ is superreflexive},
	\]
	i.e, \(X_L\) admits an equivalent uniformly convex norm.
	
	Combining these equivalences with Theorem Theorem \ref{thm:myhill-nerode} finishes the proof.
\end{proof}

\section{Applications to Formal Languages}
\label{sec:formal-language-applications}

We now apply the abstract rigidity result of the preceding section to
Banach spaces generated by derivatives of formal languages. Throughout
this section, \(\Sigma\) is a finite alphabet and \(L\subseteq\Sigma^*\)
is a language. We write
\(\mathcal D_L
=
\left\{
\mathbf 1_{w^{-1}L}:w\in\Sigma^*
\right\}\)
for the family of characteristic functions of the left derivatives of
\(L\), and
\(X_L
=
\overline{\operatorname{span}}\mathcal D_L
\subset \ell_\infty(\Sigma^*)\).

Thus \(X_L\) is the Banach space generated by the derivative orbit of
\(L\).

\begin{proposition}\label{prop:shift-invariance}
	For every \(u\in\Sigma^*\), the operator
	\(R_u:\ell_\infty(\Sigma^*)\to\ell_\infty(\Sigma^*)\), defined by
	\((R_uf)(v)=f(uv)\),
	is a linear contraction. 
	
	Moreover,
	\(R_u(X_L)\subseteq X_L\), 
	and
	\(R_u\mathbf 1_{w^{-1}L}
	=
	\mathbf 1_{(wu)^{-1}L}\).
\end{proposition}

\begin{proof}
	Fix \(u\in\Sigma^*\). It is immediate from the definition that
	\(R_u\) is linear. Moreover, for every
	\(f\in\ell_\infty(\Sigma^*)\),
	\[
	\|R_uf\|_\infty
	=
	\sup_{v\in\Sigma^*}|f(uv)|
	\leq
	\sup_{z\in\Sigma^*}|f(z)|
	=
	\|f\|_\infty.
	\]
	
	Hence
	\(\|R_u\|\leq 1\),
	so \(R_u\) is a contraction.
	
	We next consider its action on the generators of \(X_L\). Let
	\(w\in\Sigma^*\). 
	
	For every \(v\in\Sigma^*\) the next assertions are equivalent
	\begin{itemize}
		\item \(\bigl(R_u\mathbf 1_{w^{-1}L}\bigr)(v)
		=
		\mathbf 1_{w^{-1}L}(uv)=
		1\) 
		\item \(uv\in w^{-1}L\) 
		\item \(wuv\in L\) 
		\item \(v\in (wu)^{-1}L\).
	\end{itemize}
	
	Therefore
	\(R_u\mathbf 1_{w^{-1}L}
	=
	\mathbf 1_{(wu)^{-1}L}\).
	
	In particular, \(R_u\) maps every generator of \(X_L\) to another
	generator of \(X_L\). Thus
	\(R_u\left(
	\operatorname{span}
	\left\{
	\mathbf 1_{w^{-1}L}:w\in\Sigma^*
	\right\}
	\right)
	\subseteq
	\operatorname{span}
	\left\{
	\mathbf 1_{w^{-1}L}:w\in\Sigma^*
	\right\}\).
	
	Since \(R_u\) is bounded, by taking closures we get
	\(R_u(X_L)\subseteq X_L\).
\end{proof}

\begin{proposition}\label{prop:cyclic-language-space}
	The space \(X_L\) is the smallest closed subspace of
	\(\ell_\infty(\Sigma^*)\) which contains \(\mathbf 1_L\) and is
	invariant under all operators \(R_u\), \(u\in\Sigma^*\).
\end{proposition}

\begin{proof}
	By definition,
	\(X_L
	=
	\overline{\operatorname{span}}
	\left\{
	\mathbf 1_{w^{-1}L}:w\in\Sigma^*
	\right\}\).
	
	Since the empty word \(\varepsilon\) belongs to \(\Sigma^*\) and
	\(\varepsilon^{-1}L=L\), we have
	\(\mathbf 1_L\in X_L\).
	
	Moreover, by Proposition~\ref{prop:shift-invariance}, there holds
	\(R_u(X_L)\subseteq X_L\) for \(u\in\Sigma^*\).
	
	Thus \(X_L\) is a closed subspace of
	\(\ell_\infty(\Sigma^*)\) containing \(\mathbf 1_L\) and invariant
	under every \(R_u\).
	
	It remains to prove minimality. 
	
	Let
	\(Y\subseteq\ell_\infty(\Sigma^*)\) be any closed subspace such that
	\(\mathbf 1_L\in Y\)
	and
	\(R_u(Y)\subseteq Y\) for \(u\in\Sigma^*\).
	
	Since \(R_w(Y)\subseteq Y\) and \(\mathbf 1_L\in Y\), we have
	\(R_w\mathbf 1_L\in Y\) for every \(w\in\Sigma^*\).
	
	By Proposition~\ref{prop:shift-invariance}, applied with
	\(w=\varepsilon\), we have
	\(R_w\mathbf 1_L
	=
	\mathbf 1_{w^{-1}L}\).
	
	Hence
	\(\mathbf 1_{w^{-1}L}\in Y\) for \(w\in\Sigma^*\). 
	Therefore
	\(\operatorname{span}
	\left\{
	\mathbf 1_{w^{-1}L}:w\in\Sigma^*
	\right\}
	\subseteq Y\).
	
	Since \(Y\) is a closed linear subspace containing
	\(\mathbf 1_{w^{-1}L}\) for every \(w\in\Sigma^*\), it also contains
	their closed linear span. Hence,
	\(X_L\subseteq Y\).
	
	Consequently, \(X_L\) is the smallest closed subspace of
	\(\ell_\infty(\Sigma^*)\) containing \(\mathbf 1_L\) and invariant
	under all operators \(R_u\), \(u\in\Sigma^*\).
\end{proof}

\begin{proposition}\label{prop:dimension-derivatives}
	If \(L\) has exactly \(N\) distinct left derivatives, then
	\(\dim X_L\leq N\).
\end{proposition}

\begin{proof}
	Assume that \(L\) has exactly \(N\) distinct left derivatives. Then
	there exist \(w_1,\ldots,w_N\in\Sigma^*\) such that
	\(\{w^{-1}L:w\in\Sigma^*\}
	=
	\{w_1^{-1}L,\ldots,w_N^{-1}L\}\).
	
	Consequently,
	\(X_L
	=
	\overline{\operatorname{span}}
	\left\{
	\mathbf 1_{w_1^{-1}L},\ldots,
	\mathbf 1_{w_N^{-1}L}
	\right\}\).
	
	Since the span of finitely many vectors is finite-dimensional and
	therefore closed, we obtain
	\(X_L
	=
	\operatorname{span}
	\left\{
	\mathbf 1_{w_1^{-1}L},\ldots,
	\mathbf 1_{w_N^{-1}L}
	\right\}\).
	
	Hence
	\(\dim X_L\leq N\).
\end{proof}

\begin{remark}
	The inequality in Proposition~\ref{prop:dimension-derivatives} need
	not be an equality, since the characteristic functions of distinct
	left derivatives of \(L\) need not be linearly independent.
\end{remark}

\section{Dynamical Representation}
\label{sec:dynamical-representation}

We now give a dynamical realization of the spaces \(X_L\). The main
observation is that the norm structure of \(X_L\) can be recovered from
the orbit closure of the characteristic function of \(L\) under the
natural shift action of the free monoid.

Let
\(\Omega_\Sigma=\{0,1\}^{\Sigma^*}\)
be equipped with the product topology, where \(\{0,1\}\) is discrete.
Since \(\Sigma^*\) is countable, \(\Omega_\Sigma\) is compact and
metrizable.

For \(v\in\Sigma^*\), define
\(\rho_v:\Omega_\Sigma\to\Omega_\Sigma\)
by
\((\rho_vx)(u)=x(uv)\) for \(u\in\Sigma^*\).

For \(u\in\Sigma^*\), let
\(\pi_u:\Omega_\Sigma\to\mathbb R\), where
\(\pi_u(x)=x(u)\)
be the \(u\)-th coordinate function.

For a language \(L\subseteq\Sigma^*\), put
\(x_L=\mathbf 1_L\in\Omega_\Sigma\)
and define its orbit closure by
\(K_L
=
\overline{
	\{\rho_vx_L:v\in\Sigma^*\}
}
\subseteq\Omega_\Sigma\).

\begin{proposition}\label{prop:language-dynamical-system}
	The set \(K_L\) is a nonempty compact metrizable subset of
	\(\Omega_\Sigma\),
	\(\rho_v(K_L)\subseteq K_L\) for \(v\in\Sigma^*\), and \(x_L\) has a dense orbit in \(K_L\).
\end{proposition}

\begin{proof}
	By definition,
	\(K_L
	=
	\overline{
		\{\rho_wx_L:w\in\Sigma^*\}
	}
	\subseteq\Omega_\Sigma\).
	
	Since \(\rho_\varepsilon x_L=x_L\), the orbit of \(x_L\) is nonempty,
	and hence \(K_L\neq\varnothing\).
	
	The space
	\(\Omega_\Sigma=\{0,1\}^{\Sigma^*}\)
	is compact by Theorem \ref{thm:tychonoff}. Since \(\Sigma^*\) is countable,
	\(\Omega_\Sigma\) is also metrizable. As \(K_L\) is closed in
	\(\Omega_\Sigma\), it follows that \(K_L\) is compact and metrizable.
	
	We next prove shift invariance. 
	
	Fix \(v\in\Sigma^*\). The map
	\(\rho_v:\Omega_\Sigma\to\Omega_\Sigma\) is continuous in the product
	topology. 
	
	Indeed, for every \(u\in\Sigma^*\),
	\(\pi_u\circ\rho_v=\pi_{uv}\),
	and the coordinate functions are continuous.
	
	For \(v,w,u\in\Sigma^*\) and \(x\in\Omega_\Sigma\), we have
	\(\bigl(\rho_v(\rho_wx)\bigr)(u)=(\rho_wx)(uv)=x(uvw)=(\rho_{vw}x)(u)\).
		
	Therefore
	\(\rho_v\circ\rho_w=\rho_{vw}\).
	
	In particular,
	\(\rho_v
	\left(
	\{\rho_wx_L:w\in\Sigma^*\}
	\right)
	\subseteq
	\{\rho_wx_L:w\in\Sigma^*\}\).
	
	Using the continuity of \(\rho_v\), we obtain the inclusions
	\[
		\rho_v(K_L)
		=
		\rho_v
		\left(
		\overline{\{\rho_wx_L:w\in\Sigma^*\}}
		\right)\subseteq
		\overline{
			\rho_v\left(\{\rho_wx_L:w\in\Sigma^*\}\right)
		}\subseteq
		K_L.
	\]
	
	Thus \(K_L\) is invariant under every shift \(\rho_v\).
	
	Finally, the orbit of \(x_L\) is
	\(\mathcal O(x_L)
	=
	\{\rho_vx_L:v\in\Sigma^*\}\),
	and \(K_L\) was defined precisely as its closure. 
	
	Hence
	\(\overline{\mathcal O(x_L)}=K_L\),
	so \(x_L\) has a dense orbit in \(K_L\).
\end{proof}

For \(u\in\Sigma^*\), denote by
\(\pi_u^L=\pi_u|_{K_L}\)
the restriction of the \(u\)-th coordinate function to \(K_L\).

\begin{theorem}\label{thm:dynamical-representation}
	There exists a unique surjective linear isometry
	\[
	J_L:X_L\longrightarrow
	\overline{\operatorname{span}
		\{\pi_u^L:u\in\Sigma^*\}}^{\,C(K_L)}
	\]
	such that
	\(J_L\bigl(\mathbf 1_{u^{-1}L}\bigr)=\pi_u^L\) for \(u\in\Sigma^*\).
\end{theorem}

\begin{proof}
	Let
	\(\mathcal D_L
	=
	\left\{
	\mathbf 1_{u^{-1}L}:u\in\Sigma^*
	\right\}\) and let us define a linear map on the algebraic span of
	\(\mathcal D_L\) by
	\[
	J_L
	\left(
	\sum_{i=1}^n a_i\mathbf 1_{u_i^{-1}L}
	\right)
	=
	\sum_{i=1}^n a_i\pi_{u_i}^L.
	\]
	
	We will show that \(J_L\) map is well defined and isometric.
	
	For every \(u,v\in\Sigma^*\), there holds
	\[
		\pi_u^L(\rho_vx_L)=
		(\rho_vx_L)(u)=
		x_L(uv)=
		\mathbf 1_L(uv)=
		\mathbf 1_{u^{-1}L}(v).
	\]
	
	Hence, for every finite collection
	\(u_1,\ldots,u_n\in\Sigma^*\) and scalars
	\(a_1,\ldots,a_n\),
	\[
		\left\|
		\sum_{i=1}^n
		a_i\mathbf 1_{u_i^{-1}L}
		\right\|_\infty=
		\sup_{v\in\Sigma^*}
		\left|
		\sum_{i=1}^n
		a_i\mathbf 1_{u_i^{-1}L}(v)
		\right|=
		\sup_{v\in\Sigma^*}
		\left|
		\sum_{i=1}^n
		a_i\pi_{u_i}^L(\rho_vx_L)
		\right|.
	\]
	
	Since the orbit
	\(\{\rho_vx_L:v\in\Sigma^*\}\)
	is dense in \(K_L\), and the map
	\(x\longmapsto
	\sum_{i=1}^n a_i\pi_{u_i}^L(x)\)
	is continuous on \(K_L\), we obtain
	\[
	\sup_{v\in\Sigma^*}
	\left|
	\sum_{i=1}^n
	a_i\pi_{u_i}^L(\rho_vx_L)
	\right|
	=
	\sup_{x\in K_L}
	\left|
	\sum_{i=1}^n
	a_i\pi_{u_i}^L(x)
	\right|.
	\]
	
	Therefore
	\begin{equation}\label{eq:dynamical-isometry}
	\left\|
	\sum_{i=1}^n
	a_i\mathbf 1_{u_i^{-1}L}
	\right\|_\infty
	=
	\left\|
	\sum_{i=1}^n
	a_i\pi_{u_i}^L
	\right\|_{C(K_L)}.
\end{equation}
	
	In particular, if
	\(\sum_{i=1}^n
	a_i\mathbf 1_{u_i^{-1}L}=0\),
	then the right-hand side of
	\eqref{eq:dynamical-isometry} is zero. Thus
	\(\sum_{i=1}^n a_i\pi_{u_i}^L=0\),
	which shows that \(J_L\) is well defined. 
	
	Equation
	\eqref{eq:dynamical-isometry} also shows that \(J_L\) is an isometry.
	
	By definition,
	\(X_L
	=
	\overline{\operatorname{span}\mathcal D_L}\).
	
	Hence \(J_L\) extends uniquely by continuity to a linear isometry
	\(J_L:X_L\longrightarrow C(K_L)\).
	
	Since \(X_L\) is complete and \(J_L\) is an isometry,
		\(J_L(X_L)\) is a closed subspace of \(C(K_L)\).
	 Moreover, the
	range of \(J_L\) contains every \(\pi_u^L\), since
	\(J_L\bigl(\mathbf 1_{u^{-1}L}\bigr)=\pi_u^L\).
	
	Consequently,
	\(J_L(X_L)
	=
	\overline{\operatorname{span}
		\{\pi_u^L:u\in\Sigma^*\}}^{\,C(K_L)}\). 
	Thus, \(J_L\) is surjective onto
	\(\overline{\operatorname{span}}
	\{\pi_u^L:u\in\Sigma^*\}\subseteq C(K_L)\).
	
	Finally, uniqueness follows because the family
	\(\left\{
	\mathbf 1_{u^{-1}L}:u\in\Sigma^*
	\right\}\)
	has dense linear span in \(X_L\). Therefore any continuous linear
	operator satisfying
	\(J_L\bigl(\mathbf 1_{u^{-1}L}\bigr)=\pi_u^L\)
	for every \(u\in\Sigma^*\) must coincide with \(J_L\) on all of
	\(X_L\).
\end{proof}

\begin{corollary}\label{cor:transitive-subshift-characterization}
	Let \(K\subseteq\Omega_\Sigma\) be nonempty, compact, and
	shift-invariant, and suppose that \(x_0\in K\) has a dense orbit.
	Define
	\(L_{x_0}
	=
	\{u\in\Sigma^*:x_0(u)=1\}\).
	
	Then
	\begin{enumerate}
	\item \(x_{L_{x_0}}=x_0\)
	\item \(K_{L_{x_0}}=K\)
	\item There exists a surjective linear isometry
		\[J_{L_{x_0}}:X_{L_{x_0}}
		\longrightarrow
		\overline{\operatorname{span}
			\{\pi_u|_K:u\in\Sigma^*\}}^{\,C(K)}
		\]
		satisfying
		\(J_{L_{x_0}}
		\bigl(\mathbf 1_{u^{-1}L_{x_0}}\bigr)
		=
		\pi_u|_K\) for every \(u\in\Sigma^*\).
	\end{enumerate}
\end{corollary}

\begin{proof}
	By the definition of \(L_{x_0}\), for every \(u\in\Sigma^*\), the following are equivalent
	\begin{itemize}
		\item \(\mathbf 1_{L_{x_0}}(u)
	=
	1\)
	\item \(u\in L_{x_0}\)
	\item \(x_0(u)=1\).
\end{itemize}

	Since \(x_0\in\Omega_\Sigma=\{0,1\}^{\Sigma^*}\), it follows that
	\(\mathbf 1_{L_{x_0}}(u)=x_0(u)\) for \(u\in\Sigma^*\).
	
	Hence
	\(x_{L_{x_0}}
	=
	\mathbf 1_{L_{x_0}}
	=
	x_0\).
	
	By definition of the orbit closure associated with a language,
	\(K_{L_{x_0}}=\overline{\{\rho_vx_{L_{x_0}}:v\in\Sigma^*\}}\).
	
	Using \(x_{L_{x_0}}=x_0\), we obtain
	\(K_{L_{x_0}}
	=
	\overline{
		\{\rho_vx_0:v\in\Sigma^*\}
	}\).
	Since \(x_0\) has a dense orbit in \(K\), the right-hand side is
	exactly \(K\). Therefore
	\(K_{L_{x_0}}=K\).
	
	Finally, applying
	Theorem~\ref{thm:dynamical-representation} to the language
	\(L_{x_0}\), we obtain a surjective linear isometry
	\(X_{L_{x_0}}
	\longrightarrow
	\overline{\operatorname{span}
		\{\pi_u^{L_{x_0}}:u\in\Sigma^*\}}^{\,C(K_{L_{x_0}})}\).
		
	Since
	\(K_{L_{x_0}}=K\),
	we have
	\(\pi_u^{L_{x_0}}
	=
	\pi_u|_K\) for \(u\in\Sigma^*\).
	
	Consequently,
	\(X_{L_{x_0}}
	\cong
	\overline{\operatorname{span}
		\{\pi_u|_K:u\in\Sigma^*\}}^{\,C(K)}\)
	linearly and isometrically.
\end{proof}

Thus the Banach spaces generated by formal-language derivatives admit a
canonical dynamical realization: they are precisely the coordinate
Banach spaces associated with transitive binary shift systems arising
from the free-monoid action on \(\Omega_\Sigma\).

\section{Hankel Operators}
\label{sec:hankel-operators}

We next associate with each language a canonical operator whose range
encodes the Banach space \(X_L\). This provides an operator-theoretic
interpretation of regularity and leads to an exact quantitative
separation between regular and nonregular languages.

The Boolean Hankel matrix associated with \(L\) is
\((\mathbf 1_L(uv))_{u,v\in\Sigma^*}\).

Its natural realization as an operator from
\(\ell_1(\Sigma^*)\) into \(\ell_\infty(\Sigma^*)\) is given below.

\begin{definition}\label{def:language-hankel-operator}
	The language Hankel operator \(\mathcal H_L:\ell_1(\Sigma^*)\longrightarrow\ell_\infty(\Sigma^*)\)
	is the linear operator whose action on \(a\in\ell_1(\Sigma^*)\) is given,
	for every \(v\in\Sigma^*\), by
	\begin{equation}\label{eq:1}
	(\mathcal H_La)(v)
	=
	\sum_{u\in\Sigma^*}
	a(u)\mathbf 1_L(uv).
	\end{equation}
\end{definition}

The series in (\ref{eq:1}) is absolutely convergent, since
\(\sum_{u\in\Sigma^*}
|a(u)\mathbf 1_L(uv)|
\leq
\|a\|_1\).

\begin{proposition}\label{prop:hankel-range}There holds
	\begin{itemize}
		\item \(\mathcal H_L\) is bounded
		\item \(\|\mathcal H_L\|\leq1\) 
		\item for every \(u\in\Sigma^*\) there holds
	\(\mathcal H_Le_u
	=
	\mathbf 1_{u^{-1}L}\)
	\item 
	\(\overline{\operatorname{ran}\mathcal H_L}
	=
	X_L\).
\end{itemize}
\end{proposition}

\begin{proof}
	Let \(a\in\ell_1(\Sigma^*)\). For every \(v\in\Sigma^*\),
	\begin{equation}\label{eq:11}
		|(\mathcal H_La)(v)|=
		\left|
		\sum_{u\in\Sigma^*}
		a(u)\mathbf 1_L(uv)
		\right|\leq
		\sum_{u\in\Sigma^*}
		|a(u)|\mathbf 1_L(uv)\leq
		\sum_{u\in\Sigma^*}|a(u)|
		=
		\|a\|_1.
\end{equation}
	
	Taking the supremum over \(v\in\Sigma^*\) in (\ref{eq:11}), we obtain
	\(\|\mathcal H_La\|_\infty
	\leq
	\|a\|_1\).
	Hence \(\mathcal H_L\) is bounded and
	\(\|\mathcal H_L\|\leq1\).
	
	Let \(e_u\) denote the canonical unit vector of
	\(\ell_1(\Sigma^*)\) corresponding to \(u\in\Sigma^*\). Then, for
	every \(v\in\Sigma^*\),
	\[
		(\mathcal H_Le_u)(v)=
		\sum_{w\in\Sigma^*}
		e_u(w)\mathbf 1_L(wv)=
		\mathbf 1_L(uv)=
		\mathbf 1_{u^{-1}L}(v).
	\]
	
	Therefore
	\(\mathcal H_Le_u
	=
	\mathbf 1_{u^{-1}L}\). 
	It follows that every member of the generating family
		\(\{\mathbf 1_{u^{-1}L}:u\in\Sigma^*\}\)
		belongs to \(\operatorname{ran}\mathcal H_L\).
		Therefore,
		\(X_L
		=
		\overline{\operatorname{span}}
		\{\mathbf 1_{u^{-1}L}:u\in\Sigma^*\}
		\subseteq
		\overline{\operatorname{ran}\mathcal H_L}\).
	
	Taking closures gives
	\begin{equation}\label{eq:XL-in-hankel-range}
	X_L
	\subseteq
	\overline{\operatorname{ran}\mathcal H_L}.
\end{equation}
	
	For the converse inclusion, let \(a\in\ell_1(\Sigma^*)\). Choose a
	sequence of finitely supported vectors \((a_n)\) such that
	\(a_n\longrightarrow a\) in \(\ell_1(\Sigma^*)\).
	
	For each \(n\), the vector \(\mathcal H_La_n\) is a finite linear
	combination of the functions
	\(\mathbf 1_{u^{-1}L}\), and therefore
	\(\mathcal H_La_n\in X_L\).
	
	Since \(\mathcal H_L\) is bounded,
	\(\mathcal H_La_n
	\longrightarrow
	\mathcal H_La\) in \(ell_\infty(\Sigma^*)\).
	
	As \(X_L\) is closed, this implies
	\(\mathcal H_La\in X_L\).
	Thus
	\(\operatorname{ran}\mathcal H_L\subseteq X_L\),
	and consequently
	\begin{equation}\label{eq:hankel-range-in-XL}
	\overline{\operatorname{ran}\mathcal H_L}
	\subseteq X_L.
\end{equation}
	
	Combining \eqref{eq:XL-in-hankel-range} and
	\eqref{eq:hankel-range-in-XL}, we conclude that
	\(\overline{\operatorname{ran}\mathcal H_L}
	=
	X_L\).
\end{proof}

\begin{corollary}\label{cor:hankel-rank-dimension}
	If \(L\) is regular, then
	\(\operatorname{rank}\mathcal H_L
	=
	\dim X_L\).
\end{corollary}

\begin{proof}
	Since \(L\) is regular, Corollary~\ref{cor:main-language-rigidity}
	implies that \(X_L\) is finite-dimensional.
	
	By Proposition~\ref{prop:hankel-range},
	\(\overline{\operatorname{ran}\mathcal H_L}=X_L\).
	Moreover,
	\(\operatorname{ran}\mathcal H_L\subseteq X_L\).
	
	As \(X_L\) is finite-dimensional, every linear subspace of \(X_L\) is
	closed. Hence \(\operatorname{ran}\mathcal H_L\) is closed, and
	therefore
	\(\operatorname{ran}\mathcal H_L
	=
	\overline{\operatorname{ran}\mathcal H_L}
	=
	X_L\).
	
	Consequently,
	\(\operatorname{rank}\mathcal H_L
	=
	\dim\operatorname{ran}\mathcal H_L
	=
	\dim X_L\).
\end{proof}

\begin{theorem}\label{thm:compact-hankel}
	For every language \(L\subseteq\Sigma^*\), the following assertions
	are equivalent:
	\begin{enumerate}
		\item \(L\) is regular;
		\item \(\mathcal H_L\) has finite rank;
		\item \(\mathcal H_L\) is compact.
	\end{enumerate}
\end{theorem}

\begin{proof}
	\((1)\Longrightarrow(2)\) Assume first that \(L\) is regular. By
	Corollary~\ref{cor:main-language-rigidity}, the space \(X_L\) is
	finite-dimensional. Proposition~\ref{prop:hankel-range} gives
	\(\operatorname{ran}\mathcal H_L\subseteq X_L\).
	
	Hence
	\(\operatorname{rank}\mathcal H_L
	\leq
	\dim X_L
	<
	\infty\).
	Thus \(\mathcal H_L\) has finite rank.
	
	\((2)\Longrightarrow(3)\)
	Every finite-rank operator between Banach spaces is compact.
	
	\((3)\Longrightarrow(1)\)
	Suppose that \(\mathcal H_L\) is compact. Let
	\((e_u)_{u\in\Sigma^*}\) denote the canonical unit vectors of
	\(\ell_1(\Sigma^*)\). 
	
	Since
	\(\|e_u\|_1=1\) for \(u\in\Sigma^*\), the set
	\(\{e_u:u\in\Sigma^*\}\)
	is contained in the closed unit ball of \(\ell_1(\Sigma^*)\).
	Compactness of \(\mathcal H_L\) therefore implies that
	\(\{\mathcal H_Le_u:u\in\Sigma^*\}\)
	is relatively compact in \(\ell_\infty(\Sigma^*)\).
	
	By Proposition~\ref{prop:hankel-range},
	\(\mathcal H_Le_u
	=
	\mathbf 1_{u^{-1}L}\),
	and hence
	\[
	\{\mathcal H_Le_u:u\in\Sigma^*\}
	=
	\mathcal A_L
	=
	\left\{
	\mathbf 1_{u^{-1}L}:u\in\Sigma^*
	\right\}
	\subseteq
	\{0,1\}^{\Sigma^*}.
	\]
	
	By Lemma~\ref{lem:finite-valued-separated}, an infinite family of
	functions taking values in the fixed finite set \(\{0,1\}\) cannot
	be relatively compact in \(\ell_\infty(\Sigma^*)\). Consequently,
	\(\mathcal A_L\) is finite.
	
	Thus \(L\) has only finitely many distinct left derivatives. According to
	Theorem \ref{thm:myhill-nerode}, \(L\) is regular. 
\end{proof}

\begin{theorem}\label{thm:hankel-approximation-gap}
	Let \(L\subseteq\Sigma^*\).
	
	\begin{enumerate}
		\item If \(L\) is regular, then \(\mathcal H_L\) has finite rank and 
	\[
	\operatorname{dist}
	\bigl(
	\mathcal H_L,
	\mathcal F(\ell_1(\Sigma^*),\ell_\infty(\Sigma^*))
	\bigr)
	=
	\operatorname{dist}
	\bigl(
	\mathcal H_L,
	\mathcal K(\ell_1(\Sigma^*),\ell_\infty(\Sigma^*))
	\bigr)
	=
	0.
	\]	
	\item If \(L\) is nonregular, then
	\[
	\operatorname{dist}
	\bigl(
	\mathcal H_L,
	\mathcal F(\ell_1(\Sigma^*),\ell_\infty(\Sigma^*))
	\bigr)
	=
	\operatorname{dist}
	\bigl(
	\mathcal H_L,
	\mathcal K(\ell_1(\Sigma^*),\ell_\infty(\Sigma^*))
	\bigr)
	=
	\frac12
	\]
	and \(a_n(\mathcal H_L)=\frac12\) for every \(n\geq2\).
\end{enumerate}
\end{theorem}

\begin{proof}
	Suppose first that \(L\) is regular. By
	Theorem~\ref{thm:compact-hankel}, the operator \(\mathcal H_L\) has
	finite rank. Therefore,
	\(\mathcal H_L\in\mathcal F(\ell_1(\Sigma^*),\ell_\infty(\Sigma^*))\subseteq\mathcal K(\ell_1(\Sigma^*),\ell_\infty(\Sigma^*))\),
	and consequently
	\[
	\operatorname{dist}
	\bigl(
	\mathcal H_L,
	\mathcal F(\ell_1(\Sigma^*),\ell_\infty(\Sigma^*))
	\bigr)
	=
	\operatorname{dist}
	\bigl(
	\mathcal H_L,
	\mathcal K(\ell_1(\Sigma^*),\ell_\infty(\Sigma^*))
	\bigr)
	=
	0.
	\]
	
	Assume now that \(L\) is nonregular. Let
	\(\mathbf 1\in\ell_\infty(\Sigma^*)\) denote the constant function
	equal to one, and define
	\(A:\ell_1(\Sigma^*)\to\ell_\infty(\Sigma^*)\) by
	\(Aa
	=
	\frac12
	\left(
	\sum_{u\in\Sigma^*}a(u)
	\right)\mathbf 1\).
	
	The operator \(A\) has rank one. For every
	\(a\in\ell_1(\Sigma^*)\) and \(v\in\Sigma^*\), we have
	\(\bigl((\mathcal H_L-A)a\bigr)(v)=\sum_{u\in\Sigma^*}a(u)\bigl(\mathbf 1_L(uv)-\tfrac12\bigr)\).
	
	Since
	\(\left|
	\mathbf 1_L(uv)-\frac12
	\right|
	=
	\frac12\),
	we obtain
	\(\|(\mathcal H_L-A)a\|_\infty
	\leq
	\frac12
	\sum_{u\in\Sigma^*}|a(u)|
	=
	\frac12\|a\|_1\).
	
	Hence,
	\(\|\mathcal H_L-A\|\leq\frac12\).
	In fact, for every \(u\in\Sigma^*\),
	\(\|(\mathcal H_L-A)e_u\|_\infty=\|\mathbf 1_{u^{-1}L}-\tfrac12\mathbf 1\|_\infty=\tfrac12\),
	and therefore
	\(\|\mathcal H_L-A\|=\frac12\).
	
		Since \(A\) has rank one, it belongs to both the class of
		finite-rank operators and the class of compact operators, so
		\(A\in\mathcal F(\ell_1(\Sigma^*),\ell_\infty(\Sigma^*))\subseteq\mathcal K(\ell_1(\Sigma^*),\ell_\infty(\Sigma^*))\).
		By the definition of the distance from an operator to a class of
		operators, we have
		\[
		\operatorname{dist}
			\bigl(
			\mathcal H_L,
			\mathcal K(\ell_1(\Sigma^*),\ell_\infty(\Sigma^*))
			\bigr) =
			\inf_{T\in
				\mathcal K(\ell_1(\Sigma^*),\ell_\infty(\Sigma^*))}
			\|\mathcal H_L-T\|\leq
			\|\mathcal H_L-A\|
			=
			\frac12.
		\]
		
		Thus,
	
	\begin{equation}\label{eq:compact-distance-upper}
		\operatorname{dist}
		\bigl(
		\mathcal H_L,
		\mathcal K(\ell_1(\Sigma^*),\ell_\infty(\Sigma^*))
		\bigr)
		\leq
		\frac12.
	\end{equation}
	
	The same rank-one operator also gives

	\begin{equation}\label{eq:finite-rank-distance-upper}
		\operatorname{dist}
		\bigl(
		\mathcal H_L,
		\mathcal F(\ell_1(\Sigma^*),\ell_\infty(\Sigma^*))
		\bigr)
		\leq
		\|\mathcal H_L-A\|
		=
		\frac12.
	\end{equation}
	
	We now prove the reverse inequality. Since \(L\) is nonregular,
	Theorem~\ref{thm:myhill-nerode} implies that \(L\) has infinitely
	many distinct left derivatives. Hence, we can choose
	\(u_1,u_2,\ldots\in\Sigma^*\) such that
	\(u_m^{-1}L\neq u_n^{-1}L\) whenever \(m\neq n\).
	Therefore,
	\(\|\mathbf 1_{u_m^{-1}L}-\mathbf 1_{u_n^{-1}L}\|_\infty=1\) for every pair of distinct indices \(m,n\in\mathbb N\).
	
	Suppose, towards a contradiction, that there exists a compact operator
	\(T:\ell_1(\Sigma^*)\to\ell_\infty(\Sigma^*)\)
	such that
	\(\|\mathcal H_L-T\|
	=
	\delta
	<
	\frac12\).
	
	For \(m\neq n\), using
	\(\mathcal H_Le_{u_j}
	=
	\mathbf 1_{u_j^{-1}L}\),
	we obtain
	\[
	\begin{aligned}
		\|Te_{u_m}-Te_{u_n}\|_\infty
		&\geq
		\|\mathcal H_Le_{u_m}
		-\mathcal H_Le_{u_n}\|_\infty\\
		&\quad
		-\|(\mathcal H_L-T)e_{u_m}\|_\infty
		-\|(\mathcal H_L-T)e_{u_n}\|_\infty\geq
		1-2\delta
		>
		0.
	\end{aligned}
	\]
	
	Thus,
	\(\{Te_{u_n}:n\in\mathbb N\}\)
	is an infinite \((1-2\delta)\)-separated set.
	
	On the other hand, \((e_{u_n})\) is contained in the unit ball of
	\(\ell_1(\Sigma^*)\), and the compactness of \(T\) implies that
	\(\{Te_{u_n}:n\in\mathbb N\}\)
	is relatively compact. This is impossible, since a relatively compact
	metric space cannot contain an infinite set whose distinct elements
	are separated by a fixed positive distance.
	
		Therefore, there is no compact operator \(T\) satisfying
		\(\|\mathcal H_L-T\|<\frac12\). By the definition of the distance
		to the class of compact operators, it follows that
		
	\begin{equation}\label{eq:compact-distance-lower}
		\operatorname{dist}
		\bigl(
		\mathcal H_L,
		\mathcal K(\ell_1(\Sigma^*),\ell_\infty(\Sigma^*))
		\bigr)
		\geq
		\frac12.
	\end{equation}
	
	The bounds \eqref{eq:compact-distance-upper} and
	\eqref{eq:compact-distance-lower} give
		\(\operatorname{dist}\bigl(\mathcal H_L,\mathcal K(\ell_1(\Sigma^*),\ell_\infty(\Sigma^*))\bigr)=\tfrac12\).
	
	Since
	\(\mathcal F(\ell_1(\Sigma^*),\ell_\infty(\Sigma^*))
	\subseteq
	\mathcal K(\ell_1(\Sigma^*),\ell_\infty(\Sigma^*))\),
		the distance to the smaller class cannot be smaller than the distance
		to the larger class. Hence,

	\[
	\operatorname{dist}
		\bigl(
		\mathcal H_L,
		\mathcal F(\ell_1(\Sigma^*),\ell_\infty(\Sigma^*))
		\bigr)\geq
		\operatorname{dist}
		\bigl(
		\mathcal H_L,
		\mathcal K(\ell_1(\Sigma^*),\ell_\infty(\Sigma^*))
		\bigr)
		=
		\frac12.
	\]
	
		Together with the upper estimate
		\eqref{eq:finite-rank-distance-upper}, this yields
		\(\operatorname{dist}\bigl(\mathcal H_L,\mathcal F(\ell_1(\Sigma^*),\ell_\infty(\Sigma^*))\bigr)=\tfrac12\).
	
	Finally, let \(n\geq2\). By the definition of the \(n\)-th
	approximation number,
	\(a_n(\mathcal H_L)=\inf\{\|\mathcal H_L-T\|:\operatorname{rank}T<n\}\).
	
		Every operator \(T\) satisfying \(\operatorname{rank}T<n\) has
		finite rank. Therefore, the class over which the infimum defining
		\(a_n(\mathcal H_L)\) is taken is contained in
		\(\mathcal F(\ell_1(\Sigma^*),\ell_\infty(\Sigma^*))\). Consequently,

	\[
		a_n(\mathcal H_L)
		\geq
		\operatorname{dist}
		\bigl(
		\mathcal H_L,
		\mathcal F(\ell_1(\Sigma^*),\ell_\infty(\Sigma^*))
		\bigr)=
		\frac12.
	\]
	
	On the other hand,
	\(\operatorname{rank}A=1<n\) and \(\|\mathcal H_L-A\|=\frac12\).
	
	Therefore,
	\(a_n(\mathcal H_L)\leq\frac12\).
	Consequently,
	\(a_n(\mathcal H_L)=\frac12\) for every \(n\geq2\).
\end{proof}

\section{Universality}
\label{sec:universality}

The rigidity results established above show that superreflexivity is
extremely restrictive for Banach spaces generated by formal-language
derivatives. We now show that, once this restriction is removed, the
class of spaces \(X_L\) is remarkably rich. In fact, a single binary
language generates a Banach space containing an isometric copy of every
separable real Banach space.

\begin{lemma}\label{lem:binary-code-uniqueness}
	For \(i,j\in\mathbb N_0\), define
	\(U_i=1^i0\) and \(V_j=0^j1\).
	Then 
	\begin{enumerate}
		\item the family
	\(\{U_i:i\in\mathbb N_0\}\)
	is prefix-free 
	\item \(U_iw=U_kV_j\)
	if and only if \(i=k\) and \(w=V_j\).
\end{enumerate}
\end{lemma}

\begin{proof}
	1. We first show that the family
	\(\{U_i:i\in\mathbb N_0\}\)
	is prefix-free. 
	
	Let \(i,k\in\mathbb N_0\) with \(i<k\). Then
	\(U_i=1^i0\) and \(U_k=1^k0\).
	At position \(i+1\), the word \(U_i\) contains the symbol \(0\),
	whereas \(U_k\) contains the symbol \(1\). Hence \(U_i\) is not a
	prefix of \(U_k\). Since \(U_k\) is longer than \(U_i\), \(U_k\)
	cannot be a prefix of \(U_i\). Thus no two distinct words in the
	family are prefixes of one another.
	
	2. Now suppose that
	\(U_iw=U_kV_j\)
	Both \(U_i\) and \(U_k\) are prefixes of the same word. Since the
	family \(\{U_i:i\in\mathbb N_0\}\) is prefix-free, it follows that
	\(i=k\)
	Therefore
	\(U_iw=U_iV_j\).
	
Since
	\(U_iw=U_iV_j\),
	the left cancellation law for concatenation of words yields
	\(w=V_j\).
	
	Conversely, if \(i=k\) and \(w=V_j\), then clearly
	\(U_iw=U_kV_j\).
\end{proof}

\begin{theorem}\label{thm:universal-binary-language}
	There exists a language
	\(L_{\mathrm{univ}}\subseteq\{0,1\}^*\)
	such that \(X_{L_{\mathrm{univ}}}\) contains a closed subspace
	linearly isometric to \(C(\Delta)\), where
	\(\Delta=\{0,1\}^{\mathbb N}\)
	is the Cantor space.
\end{theorem}

\begin{proof}
	Let \((C_i)_{i\in\mathbb N_0}\) be an enumeration of the 
	subsets of \(\Delta\) that are both closed and open, and let
	\(\{z_j\}_{j\in\mathbb N_0}\)
	be a dense sequence in \(\Delta\). Such an enumeration exists because
	the family of subsets of the Cantor space, which a simultaneously closed and open is countable.
	
	For \(i,j\in\mathbb N_0\), let
	\(U_i=1^i0\), \(V_j=0^j1\), and define
	\[
	L_{\mathrm{univ}}
	=
	\left\{
	U_iV_j:
	i,j\in\mathbb N_0,\;
	z_j\in C_i
	\right\}
	\subseteq\{0,1\}^*.
	\]
	
	We first determine the left derivatives corresponding to the words
	\(U_i\). For \(i,j\in\mathbb N_0\), there holds
	\begin{equation}\label{eq:universal-code-values}
		\mathbf 1_{U_i^{-1}L_{\mathrm{univ}}}(V_j)
		=
		\mathbf 1_{L_{\mathrm{univ}}}(U_iV_j)=
		\mathbf 1_{C_i}(z_j).
	\end{equation}
	
	Moreover, if \(w\notin\{V_j:j\in\mathbb N_0\}\), then
	\(\mathbf 1_{U_i^{-1}L_{\mathrm{univ}}}(w)=0\).
	
	Indeed, if \(U_iw\in L_{\mathrm{univ}}\), then for some
	\(k,j\in\mathbb N_0\), holds the equality
	\begin{equation}\label{eq:3}
	U_iw=U_kV_j.
	\end{equation}
	By Lemma~\ref{lem:binary-code-uniqueness} and (\ref{eq:3}), we get that
	\(i=k\) and \(w=V_j\).
	
	Let
	\(\mathcal A
	=
	\operatorname{span}
	\left\{
	\mathbf 1_{C_i}:i\in\mathbb N_0
	\right\}
	\subseteq C(\Delta)\).
	
	The linear span of the characteristic functions of subsets of
	the Cantor space both closed and open is dense in \(C(\Delta)\). Thus
	\(\overline{\mathcal A}^{\,C(\Delta)}
	=
	C(\Delta)\).
	
	We define a linear map on \(\mathcal A\) formally by
	\[
	T
	\left(
	\sum_{i=0}^n a_i\mathbf 1_{C_i}
	\right)
	=
	\sum_{i=0}^n
	a_i\mathbf 1_{U_i^{-1}L_{\mathrm{univ}}}.
	\]
	
	We verify simultaneously that \(T\) is well defined and isometric.
	
	Let
	\(f=\sum_{i=0}^n a_i\mathbf 1_{C_i}\in\mathcal A\).
	Using \eqref{eq:universal-code-values}, we obtain
	\[
		\left\|
		Tf
		\right\|_\infty
		=
		\sup_{w\in\{0,1\}^*}
		\left|
		\sum_{i=0}^n
		a_i\mathbf 1_{U_i^{-1}L_{\mathrm{univ}}}(w)
		\right|=
		\sup_{j\in\mathbb N_0}
		\left|
		\sum_{i=0}^n
		a_i\mathbf 1_{C_i}(z_j)
		\right|=
		\sup_{j\in\mathbb N_0}|f(z_j)|.
	\]
	
	Since \((z_j)\) is dense in \(\Delta\) and \(f\) is continuous, we have
	\(\sup_{j\in\mathbb N_0}|f(z_j)|=\sup_{z\in\Delta}|f(z)|=\|f\|_{C(\Delta)}\).
	
	Consequently,
	\begin{equation}\label{eq:universal-isometry}
	\|Tf\|_\infty
	=
	\|f\|_{C(\Delta)}.
\end{equation}

	In particular, if a finite linear combination of the functions
	\(\mathbf 1_{C_i}\) represents the zero function, then the
	corresponding linear combination of
	\(\mathbf 1_{U_i^{-1}L_{\mathrm{univ}}}\) also vanishes. Hence \(T\)
	is well defined. Equation~\eqref{eq:universal-isometry} shows that
	\(T\) is a linear isometry.
	
	For every \(i\in\mathbb N_0\),
	\(\mathbf 1_{U_i^{-1}L_{\mathrm{univ}}}
	\in
	X_{L_{\mathrm{univ}}}\)
	and therefore
	\(T(\mathcal A)\subseteq X_{L_{\mathrm{univ}}}\).
	
	Since \(\mathcal A\) is dense in \(C(\Delta)\), the isometry \(T\)
	extends uniquely to a linear isometry
	\(\widetilde T:
	C(\Delta)
	\longrightarrow
	X_{L_{\mathrm{univ}}}\).
	The range of an isometry from a Banach space is closed. Hence
	\(\widetilde T(C(\Delta))\)
	is a closed subspace of \(X_{L_{\mathrm{univ}}}\) linearly isometric
	to \(C(\Delta)\).
	
	Thus \(X_{L_{\mathrm{univ}}}\) contains a closed subspace linearly
	isometric to \(C(\Delta)\).
\end{proof}

\begin{corollary}\label{cor:universal-separable-banach}
	There exists a binary language
	\(L_{\mathrm{univ}}\subseteq\{0,1\}^*\) such that every separable real
	Banach space admits a linear isometric embedding into
	\(X_{L_{\mathrm{univ}}}\).
\end{corollary}

\begin{proof}
	By Theorem~\ref{thm:banach-mazur}, every separable real Banach space
	admits a linear isometric embedding into \(C(\Delta)\). By
	Theorem~\ref{thm:universal-binary-language}, \(C(\Delta)\) admits a
	linear isometric embedding into \(X_{L_{\mathrm{univ}}}\).
	Composing the two embeddings gives the result.
\end{proof}

\begin{remark}
	Since \(X_{L_{\mathrm{univ}}}\subseteq\ell_\infty(\{0,1\}^*)\),
	the preceding universality result is consistent with the classical
	universality of \(\ell_\infty\). The point here is more specific:
	universality is already attained by a separable closed subspace
	generated by the \(\{0,1\}\)-valued left derivatives of a single
	binary language.
\end{remark}

Thus the derivative-generated construction exhibits a sharp geometric
contrast. Superreflexivity forces \(X_L\) to be finite-dimensional,
whereas in the absence of this restriction a single binary language can
generate a space universal for the class of separable real Banach spaces.

\section*{Acknowledgments}

The third author is partially supported by the Bulgarian National Science
Fund (BNSF), Grant No.~KP-06-H92/6. The fourth author is partially supported
by the Bulgarian National Science Fund under Grant No.~KP-06-N92/1.

\section*{Declarations}

\paragraph*{Funding.}

The third author is partially supported by the Bulgarian National Science
Fund (BNSF), Grant No.~KP-06-H92/6. The fourth author is partially supported
by the Bulgarian National Science Fund under Grant No.~KP-06-N92/1.

\paragraph*{Conflict of interest.}

The authors declare that they have no conflict of interest.

\paragraph*{Data availability.}

Data sharing is not applicable to this article, as no datasets were generated
or analyzed during the current study.

\end{document}